\documentclass[10pt]{amsart} 
\usepackage{amsmath}
\usepackage{amssymb}
\usepackage{mathtools}      
\usepackage{mathabx} 
\usepackage{enumitem}
\usepackage{amsthm}
\usepackage{thmtools}
\usepackage{xcolor} 
\usepackage{tikz}
\usepackage[font=small]{caption}
\usepackage[colorlinks=true,backref=page,bookmarksopen=true]{hyperref} 
\usepackage[capitalise]{cleveref} 

\hypersetup{
    colorlinks,
    linkcolor={red!50!black},
    citecolor={green!50!black},
    urlcolor={blue!80!black}
}

\renewcommand*{\backref}[1]{}

\renewcommand*{\backrefalt}[4]{\ \tiny 
  \ifcase #1 ({\color{red} \bf NOT CITED.})%
  \or    ($\uparrow$#2)%
  \else   ($\uparrow$#2)%
  \fi}
  
\declaretheorem[numberwithin=section]{theorem}

\declaretheorem[sibling=theorem]{lemma}

\declaretheorem[sibling=theorem, style=definition]{definition}
\declaretheorem[sibling=theorem, style=definition]{example}

\declaretheorem[sibling=theorem, style=remark]{remark}

\hypersetup{bookmarksdepth = 3} 

\setlist[enumerate,1]{label={\upshape(\alph*)},ref=\alph*}

 \newcommand{\Z}{\mathbb{Z}}  \newcommand{\R}{\mathbb{R}} \newcommand{\C}{\mathbb{C}}

\newcommand{\st}{\;\mathord{:}\;}

\DeclareMathOperator{\diam}{diam}

\renewcommand{\setminus}{\smallsetminus}
\renewcommand{\emptyset}{\varnothing}
\renewcommand{\epsilon}{\varepsilon}

\newcommand{\arxiv}[2]{\href{http://arxiv.org/abs/#1}{arXiv: {#1} [{#2}]}}
\newcommand{\doi}[1]{\href{http://doi.org/#1}{\tt doi}} 
\newcommand{\directlink}[1]{\href{#1}{\tt URL}}
\newcommand{\MRev}[1]{\href{https://mathscinet.ams.org/mathscinet-getitem?mr=#1}{\tt MR}} 
\newcommand{\Zbl}[1]{\href{https://zbmath.org/?q=an:#1}{\tt Zbl}} 

\DeclareFontFamily{U} {MnSymbolA}{} 
\DeclareFontShape{U}{MnSymbolA}{m}{n}{
   <-6> MnSymbolA5
   <6-7> MnSymbolA6
   <7-8> MnSymbolA7
   <8-9> MnSymbolA8
   <9-10> MnSymbolA9
   <10-12> MnSymbolA10
   <12-> MnSymbolA12}{}
\DeclareFontShape{U}{MnSymbolA}{b}{n}{
   <-6> MnSymbolA-Bold5
   <6-7> MnSymbolA-Bold6
   <7-8> MnSymbolA-Bold7
   <8-9> MnSymbolA-Bold8
   <9-10> MnSymbolA-Bold9
   <10-12> MnSymbolA-Bold10
   <12-> MnSymbolA-Bold12}{}
\DeclareSymbolFont{MnSyA} {U} {MnSymbolA}{m}{n}
\DeclareFontFamily{U} {MnSymbolC}{}
\DeclareFontShape{U}{MnSymbolC}{m}{n}{
  <-6> MnSymbolC5
  <6-7> MnSymbolC6
  <7-8> MnSymbolC7
  <8-9> MnSymbolC8
  <9-10> MnSymbolC9
  <10-12> MnSymbolC10
  <12-> MnSymbolC12}{}
\DeclareFontShape{U}{MnSymbolC}{b}{n}{
  <-6> MnSymbolC-Bold5
  <6-7> MnSymbolC-Bold6
  <7-8> MnSymbolC-Bold7
  <8-9> MnSymbolC-Bold8
  <9-10> MnSymbolC-Bold9
  <10-12> MnSymbolC-Bold10
  <12-> MnSymbolC-Bold12}{}
\DeclareSymbolFont{MnSyC} {U} {MnSymbolC}{m}{n}

\DeclareMathSymbol{\top}{\mathord}{MnSyA}{219} 
\DeclareMathSymbol{\bot}{\mathord}{MnSyA}{217}
\DeclareMathSymbol{\smallplus}{\mathord}{MnSyC}{20} 
\DeclareMathSymbol{\smallminus}{\mathord}{MnSyC}{16} 
\DeclareMathSymbol{\smalltimes}{\mathord}{MnSyC}{21}
\DeclareMathSymbol{\smallpm}{\mathord}{MnSyC}{22} 
\DeclareMathSymbol{\smallmp}{\mathord}{MnSyC}{23}

\begin{document}

\title{Removing quasiconformal orbits}
\author{Jairo Bochi}
\address{Department of Mathematics, The Pennsylvania State University}
\email{\href{mailto:bochi@psu.edu}{bochi@psu.edu}}
\date{November 12, 2025}

\maketitle

\begin{abstract}
We show that generic continuous linear cocycles over shifts and other zero-dimensional systems admit no quasiconformal orbits, thus providing a partial answer to a question of Nassiri, Rajabzadeh, and Reshadat. The proof relies on a new result about towers for homeomorphisms of zero-dimensional spaces, which may be of independent interest. 
\end{abstract}


\section{Introduction}

Linear cocycles model compositions of linear maps that are driven by an underlying dynamical system. 
They have been studied from various perspectives and are useful in diverse contexts: see the textbooks \cite{LArnold,BP,ColoniusK,DamanikF,JohnEtc,Viana}.
This paper deals with discrete-time continuous linear cocycles in the topological (continuous) setting.
We now provide a precise definition.  

Let $T \colon X \to X$ be a homeomorphism of a compact metric space $X$.
Let $\mathbb{F}$ be either the field $\R$ or the field $\C$.
Let $d$ be a positive integer. 
A \emph{$d$-dimensional continuous linear cocycle} over $T$ is a continuous map 
$(x,n) \in X \times \Z \mapsto F^{(n)}(x) \in \mathrm{GL}(d,\mathbb{F})$ such that 
\begin{equation}
	F^{(m+n)}(x) = F^{(m)}(T^n x) F^{(n)}(x)  \quad \text{\emph{(cocycle identity)}.}
\end{equation}
Note that the cocycle is uniquely determined by the map $F \coloneq F^{(1)}$.
For this reason, we also refer to the pair $(T,F)$ as a linear cocycle. 
The set of linear cocycles over $T$ can be identified with the space $C^0(X,\mathrm{GL}(d,\mathbb{F}))$, which we endow with the complete metric 
\begin{equation}\label{e.def_metric}
	d(F,G) \coloneq \sup_{x \in X} \left( \|F(x)-G(x)\| + \|F(x)^{-1}-G(x)^{-1} \| \right) \, ,
\end{equation}
For definiteness, $\|A\|$ will denote the Euclidean-induced operator norm of a matrix~$A$; 
equivalently, $\|A\| = \sigma_1(A)$, where $\sigma_1(A) \ge \cdots \ge \sigma_d(A)$ denote the singular values of $A$.

Let $(T,F)$ be a continuous linear cocycle as above, and fix an integer $k$ with $0<k<d$.
Let us say that the cocycle admits an \emph{exponential gap} between $k$-th and $k+1$-th singular values if there exist positive constants $c$ and $\tau$ such that, for all $x \in X$ and $n > 0$,
\begin{equation}
\sigma_k(F^{(n)}(x)) > c e^{\tau n} \sigma_{k+1}(F^{(n)}(x)) \, .
\end{equation}
As proved in \cite{BG}, this condition is satisfied if and only if the cocycle admits a \emph{dominated splitting} with a dominating bundle of dimension $k$ -- we refer to the article for the exact definition. 
Dominated splittings play an important role in many problems in dynamics: see \cite{BG,BPS,NRR,Sambarino} and references therein.

The \emph{condition number} of a matrix $A \in \mathrm{GL}(d,\mathbb{F})$ is defined as
\begin{equation}\label{e.def_kappa}
	\kappa(A) \coloneq \|A\| \, \|A^{-1}\|  = \frac{\sigma_1(A)}{\sigma_d(A)} \, .
\end{equation}
Note that $\kappa(A) \ge 1$, with equality if and only if $A$ preserves the standard inner product up to a scalar factor. 
Following Nassiri, Rajabzadeh, and Reshadat \cite{NRR}, 
given a continuous linear cocycle as above, 
we say that a point $x \in X$ has a \emph{quasiconformal orbit} if
\begin{equation}\label{e.def_qc_orbit}
	\sup_{n \in \Z} \kappa(F^{(n)}(x)) < \infty \, .
\end{equation}

It is clear that a cocycle admitting a dominated splitting cannot have quasiconformal orbits. 
It is also clear that the converse statement is false, that is, a cocycle can fail to have a dominated splitting and at the same time also fail to have quasiconformal orbits: consider for example a cocycle generated by non-trivial unipotent matrices. 
Nevertheless, in some situations, the converse statement is true up to perturbation: see \cite[Theorem~A]{NRR} (see also \cite[Proposition~2.1]{BDP}, \cite[Theorem~3]{BRams} for related results).

Given a continuous cocycle $(T,F)$ taking values in the group $\mathrm{GL}(2,\R)$, if $p \in X$ is a periodic point of period $\ell$ and the matrix $F^{(\ell)}(p)$ has non-real eigenvalues, then the orbit of $p$ is quasiconformal.
Furthermore, this property will persist under sufficiently small perturbations of $F$ in the $C^0$ topology.
Thus, existence of quasiconformal orbits can be $C^0$-robust.
However, this reasoning breaks down if the dimension is at least $3$, or if the field is $\C$. 

If the base dynamics is a full shift, Nassiri, Rajabzadeh, and Reshadat \cite[Theorem~B]{NRR} show that for any dimension $d \ge 2$, H\"older continuous $\mathrm{GL}(d,\R)$-cocycles possessing quasiconformal orbits form a set of nonempty interior with respect to the appropriate H\"older topology. They also ask whether this robustness is possible in the $C^0$ topology (excluding the trivial $2$-dimensional case explained above): see \cite[Question~8.2]{NRR}.
The goal of the present note is to answer this question in the negative, at least under some additional hypotheses on the topological dynamical system $(X,T)$. 

Let us recall that a topological space has \emph{zero topological dimension} if every open cover can be refined to an open partition. 
In this paper, we work with compact metric spaces, in which case having zero topological dimension is equivalent to the existence of a basis of the topology formed by clopen sets, among several other characterizations: see \cite[Theorem~2.9.1]{Coor} and \cite[Theorem~2.10]{Putnam}.

Our result is as follows:

\begin{theorem}\label{t.main}
Let $X$ be a compact metric space of zero topological dimension.
Let $T \colon X \to X$ be a homeomorphism with finitely many periodic points of any given period. 
Let $\mathbb{F}$ be either $\C$ or $\R$, and let $d \ge 2$.
Additionally, assume that $d \ge 3$ if $\mathbb{F} = \R$ and $T$ has periodic points. 
Let $\mathcal{QC}$ be the set of $F \in C^0(X, \mathrm{GL}(d,\mathbb{F}))$ such that the corresponding cocycle admits quasiconformal orbits.
Then the set $\mathcal{QC}$ is meager in $C^0(X, \mathrm{GL}(d,\mathbb{F}))$.
\end{theorem}

\Cref{t.main} applies, for instance, to any subshift on a finite alphabet.

\begin{remark}
The conclusions of \cref{t.main} hold true if instead of assuming that $X$ is zero-dimensional, we assume that $T$ is minimal and $X$ is infinite: see \cite[Section~3.2]{NRR}.
\end{remark}

\begin{remark}
Following \cite{NRR}, we say that a cocycle $(T,F)$ has a \emph{bounded orbit} if $\sup_{n \in \Z} \| F^{(n)}(x) \| < \infty$ for some $x \in X$.
It readily follows from \cref{t.main} that, under the same hypotheses, the $\mathrm{SL}(d,\mathbb{F})$-valued cocycles admitting 
bounded orbits forms a meager subset of $C^0(X, \mathrm{SL}(d,\mathbb{F}))$.
By contrast, the existence of \emph{bounded vectorial orbits} (that is, the existence of $x\in X$ and $v \in \R^d \setminus \{0\}$ such that $\sup_{n \in \Z} \| F^{(n)}(x) v \| < \infty$) is essentially equivalent to failure of uniform hyperbolicity, by results of Ma\~n\'e and Sacker--Sell \cite[Theorems~6.51, 6.52]{ChiLat} (see also \cite[Theorem~3.8.2]{DamanikF} for the $\mathrm{SL}(2,\R)$ case). 
Therefore, existence of bounded vectorial orbits can have non-empty interior in $C^0(X, \mathrm{GL}(d,\mathbb{F}))$.
\end{remark}

To motivate the proof of \cref{t.main}, note that breaking the quasiconformality of any individual orbit would be straightforward if we were allowed to perturb the matrices along that orbit independently. The difficulty lies in achieving this simultaneously for all orbits through a coherent and continuous perturbation of the cocycle. As one might anticipate, our approach relies on a \emph{tower construction}, combined with a simple Baire argument. 

Towers are used in dynamics in innumerable ways, starting with the Rokhlin--Halmos lemma (see e.g.\ \cite{Kornfeld}). Towers are particularly useful in low-regularity constructions like the one in the present paper. 

Towers are inimical to periodic points: a tower of height $n$ cannot contain any periodic points of period $n$ or less. The Rokhlin--Halmos lemma assumes that periodic points form a set of measure zero, and thus can be ignored. 
In our setting, as we aim for a uniform and not an almost-everywhere statement, we cannot disregard the periodic points. Furthermore, the nature of the problem makes periodic orbits somewhat troublesome.

\medskip

We overcome the difficulties above by proving a new type of tower lemma.
Informally, we find partitions of the whole space $X$ by (clopen) towers with the following crucial property.
Short towers, meaning those of height less than a given threshold $N$, actually behave as if they were tall: any orbit, upon first entering a short tower, is \emph{captured} and must spend at least $N$ time units in the tower before exiting.
This is the content of \cref{t.dungeon} below, which is a result in topological dynamics (without any direct relation to linear cocycles), and may be of independent interest.

The remainder of this paper is organized as follows: 
\Cref{s.castles} discusses the basics of towers and castles. 
\Cref{s.local_capture} establishes a capturing property around periodic points.
The proof of the new tower lemma is given in \cref{s.dungeon}.
These three sections focus exclusively on topological dynamics.
In \cref{s.eradicate} we return to linear cocycles and prove \cref{t.main}.
In the final \cref{s.future}, we briefly discuss the possibility of improving our theorems and provide references for some related results.

\medskip
\noindent\textbf{Notations:} The symbols $\sqcup$, $\bigsqcup$ denote disjoint unions.  Double brackets denote intervals of integers, that is, if $m,n \in \Z$, then $\ldbrack m,n \rdbrack \coloneq \{ k \in \Z \st m \le k \le n\}$.

\section{Towers and castles}\label{s.castles}

\begin{definition}\label{def.tower_castle}
Let $T \colon X \to X$ be a bijective map.
Suppose $\ell \ge 1$ and $B \subseteq X$ are such that $B \cap T^{-i} B = \emptyset$ for all $i \in \ldbrack 0, \ell-1 \rdbrack$. 
Then the set
\begin{equation}
	K = \bigsqcup_{i=0}^{\ell-1} T^i B
\end{equation}
is called a \emph{tower} with \emph{base} $B$ and \emph{height} $\ell$ with respect to the dynamics $T$, and each of the (disjoint) sets $B, TB, \dots, T^{\ell-1} B$ is called a \emph{floor} of the tower. 
A \emph{castle} is any disjoint union of towers, and its \emph{base} is the union of their bases.
When a castle covers the whole space we call it a \emph{castle partition}.
\end{definition}

Castles are also called \emph{skyscrapers} or \emph{multitowers} in the literature. 
We follow the terminology of \cite[p.~39]{Putnam}.

In this paper, we are interested in constructing castles composed of finitely many towers. 
For that reason, the following defintion will prove itself useful:

\begin{definition}\label{def.feedback}
Let $T \colon X \to X$ be a bijective map.
A subset $E \subseteq X$ is called \emph{feedback set} with respect to $T$ if there exists $n_0 > 0$ such that every segment of orbit of length $n_0$ hits $E$, that is,
\begin{equation}\label{e.feedback}
	\bigcup_{n=0}^{n_0-1} T^{-n} E = X \, .
\end{equation}
\end{definition}

Equivalently, $E$ is a feedback set if and only if for every $x \in X$, the set of times $n \in \Z$ such that $T^n x \in E$ has bounded gaps (i.e., it is \emph{syndetic} \cite[p.~28, Definition 1.7]{Fu81}). A feedback set is also called a \emph{marker set}: see \cite[p.~98, Definition~2]{Down}.
A related but distinct notion in measurable dynamics is that of a \emph{sweep-out set}: see \cite[p.~162]{AlpPra}.
Our terminology is motivated by the following example:

\begin{example}
Suppose $T$ is a two-sided subshift of finite type (not necessarily of memory $1$).
Then $T$ can be represented as the edge shift of a higher edge  graph $G$: see \cite[Theorem~2.3.2]{LindMarcus}. 
Suppose $F$ is a set of edges of $G$ which is a feedback set in the sense of graph theory \cite[Chapter~15]{Bang}, that is, the removal of those edges makes the graph acyclic. 
Let $E\subseteq X$ be the union of all cylinders corresponding to the edges in $F$; then $E$ is a feedback set in the sense of \cref{def.feedback}.
The same construction applies to any homeomorphism of a zero-dimensional space, using graph covers \cite{BBK}.
\end{example}

The following is a finitary version of the Kakutani--Rokhlin skyscraper construction:

\begin{lemma}\label{l.Kakutani}
Let $T \colon X \to X$ be a bijective map and let $E \subseteq X$ be a feedback set. 
Then there exists a partition of $X$ into finitely many towers forming a castle with base $E$.
Additionally, if $T$ is a homeomorphism of a topological space, and $E$ is a clopen subset, then the castle partition above can be chosen in such a way that each floor of each tower is a clopen set. 
\end{lemma}

\begin{proof}
Take $n_0>0$ such that property \eqref{e.feedback} holds.
For each $\ell \in \ldbrack 1, n_0 \rdbrack$, let
\begin{equation}
B_\ell \coloneq E \cap T^{-\ell}(E) \setminus \bigcup_{i=1}^{\ell-1} T^{-i}(E) \, ,
\end{equation}
that is, the set of points in $E$ whose first return to $E$ occurs at time $\ell$.
These sets form a partition of $E$. 
Furthermore, each set $B_\ell$ is the base of a tower of height $\ell$, and those towers form a partition of $X$.
If $T$ is a homeomorphism and $E$ is a clopen set, then each set $B_\ell$ is clopen.
\end{proof}

The following is a simple but useful observation:

\begin{lemma}\label{l.open_feedback}
Suppose $T \colon X \to X$ is a homeomorphism of a compact metric space~$X$.
An open set $E \subseteq X$ is a feedback set if and only if every (two-sided) orbit hits~$E$, that is,
\begin{equation}
	\bigcup_{n \in \Z} T^{-n} E = X \, .
\end{equation}
\end{lemma}

\begin{proof}
Immediate.
\end{proof}

Suppose we want to find a castle partition of the space $X$ into uniformly high towers.
There is an obstruction: $T$ cannot have periodic points of small period.
The following \lcnamecref{l.highcastle} shows that this obvious necessary condition is also sufficient, at least when the ambient space is zero-dimensional:

\begin{lemma}[Downarowicz]\label{l.highcastle}
Let $X$ be a compact metric space of zero topological dimension and let $T \colon X \to X$ be a homeomorphism.
Let $N>0$ and suppose that $T$ has no periodic points of period less than $N$.
Then there exists a castle partition of $X$ into finitely many towers of height $\ge N$ whose floors are clopen sets.
\end{lemma}

Downarowicz original statement \cite[p.~98, Lemma~1]{Down} yields the conclusion of \cref{l.highcastle} under the stronger hypothesis that $T$ has no periodic points at all; nevertheless his argument proves the statement above. Downarowicz also allows $T$ to be non-invertible, which requires extra care. For the convenience of the reader we will provide a proof of the statement we need. 

\begin{proof}[Proof of \cref{l.highcastle}]
Suppose $T \colon X \to X$ is a homeomorphism without periodic points of period less than $N$, where $X$ is a zero-dimensional compact metric space $X$. 
Then each point in $X$ admits a clopen neighborhood $V$ which is disjoint from $TV,\dots,T^{N-1} V$. By compactness, we can cover $X$ by finitely many neighborhoods $V_1,\dots,V_m$ of this type. 
Define a nested sequence of sets $B_0 \subseteq B_1 \subseteq \cdots \subseteq B_m$ recursively as follows: $B_0 \coloneq \emptyset$ and 
\begin{align}
B_{j+1} &\coloneq B_j \cup ( V_{j+1} \setminus E_j) \, , \quad\text{where} \\ 
E_j &\coloneq \bigcup_{n=-N+1}^{N-1} T^n B_j \, .
\end{align}

We claim that 
\begin{equation}\label{e.claim1}
B_j \cap T^n B_j = \emptyset \text{ for all $j \in \ldbrack 0, m\rdbrack$ and $n \in \ldbrack 0, N-1 \rdbrack$.}
\end{equation}
The claim is proved by induction on $j$: it is clearly true for $j=0$, and the induction step goes as follows:
\begin{equation}
\begin{aligned}
B_{j+1} \cap T^n B_{j+1}  
&= \big(B_j \cup ( V_{j+1} \setminus E_j)\big) \cap T^n \big(B_j \cup ( V_{j+1} \setminus E_j)\big) \\ 
&\subseteq (B_j \cap T^n B_j) \cup (B_j \setminus T^n E_j) \cup (E_j \setminus T^n B_j) \cup (B_j \cap T^n B_j) \\
&= \emptyset \, .
\end{aligned}
\end{equation}
This proves \eqref{e.claim1}.

Next, note that for each $j \in \ldbrack 1, m\rdbrack$,
\begin{equation}
V_j \subseteq \underbrace{(V_j \setminus E_{j-1})}_{\subseteq B_j} \cup \underbrace{E_{j-1}}_{\subseteq E_j} \subseteq E_j \, .
\end{equation}
It follows that
\begin{equation}
E_m = \bigcup_{j=1}^m E_j \supseteq \bigcup_{j=1}^m V_j = X \, .
\end{equation}
In particular, $B_m$ is a feedback set. 
Consider the Kakutani--Rokhlin castle partition of $X$ with base $B_m$, given by \cref{l.Kakutani}.
Each tower has height at least $N$, by \eqref{e.claim1}. 
Since $B_m$ is a clopen set, all floors of all towers are clopen sets. 
\end{proof}

\section{Capturing property and periodic orbits}\label{s.local_capture}

We now introduce the following crucial notion:

\begin{definition}\label{def.capturing}
Let $T$ be a homeomorphism of a compact metric space $X$, and let $N$ be a positive integer. 
A set $E \subseteq X$ is called \emph{$N$-capturing} with respect to $T$ if 
\begin{equation}
E \setminus TE \subseteq \bigcap_{n=0}^{N-1} T^{-n} E  \, .
\end{equation}
\end{definition}

In other words, the set $E$ is $N$-capturing if every orbit, upon first entering $E$, remains in $E$ at least $N-1$ additional time units before exiting.
Another formulation is follows: $E$ is $N$-capturing if and only if
\begin{equation}\label{e.capture_reformulated}
\forall x \in E \ \exists m,n \ge 0 \text{ s.t.\ }
m+n+1 \ge N \text{ and }
\forall i \in \ldbrack -m, n \rdbrack, \ 
T^{i} x \in E \, .
\end{equation}
In particular, $E$ is also $N$-capturing with respect to $T^{-1}$.

\smallskip

Every tower of height $\ell \ge N$ is an $N$-capturing set. 
(On the other hand, a tower of height $\ell < N$ which is a $N$-capturing set is actually $\ell \lceil N/\ell \rceil$-capturing).
Thus, non-periodic points have $N$-capturing neighborhoods, where $N$ can be chosen arbitrarily large.
It turns out that the same is true for periodic points, as a consequence of the following lemma:

\begin{lemma}\label{l.capturingtower}
Let $X$ be a compact metric space and let $T \colon X \to X$ be a homeomorphism.
Suppose $p$ is a periodic point of least period $\ell$.
Let $N$ be a positive integer and let $U$ be a neighborhood of $p$.
Then there exists a neighborhood $B$ of $p$  with $B \subseteq U$ such that the sets $B,TB,\dots,T^{\ell-1}B$ are disjoint, their union
\begin{equation}\label{e.tower_around_orb_p}
K \coloneq \bigsqcup_{i=0}^{\ell-1} T^i B \, ,
\end{equation}
is a $N$-capturing set, and its complement $K^\mathsf{c}$ is also a $N$-capturing set.
Furthermore, if $X$ is zero-dimensional, then $B$ can be chosen as a clopen set. 
\end{lemma}

The statement above should be evident if $p$ is a hyperbolic periodic point of a diffeomorphism, for instance: see \cref{f.bicapturing}. 

We will prove \cref{l.capturingtower} initially for fixed points ($\ell=1$), and later deduce the general statement from this particular case. 

\begin{figure}
	\begin{tikzpicture}[scale=1.7]
		\fill[black!30] (0.045,0.896)--(0.055,0.734)--(0.067,0.601)--(0.081,0.492)--(0.099,0.403)--(0.121,0.330)--(0.148,0.270)--(0.181,0.221)--(0.221,0.181)--(0.270,0.148)--(0.330,0.121)--(0.403,0.099)--(0.492,0.081)--(0.601,0.067)--(0.734,0.055)--(0.896,0.045)--(0.896,-0.045)--(0.734,-0.055)--(0.601,-0.067)--(0.492,-0.081)--(0.403,-0.099)--(0.330,-0.121)--(0.270,-0.148)--(0.221,-0.181)--(0.181,-0.221)--(0.148,-0.270)--(0.121,-0.330)--(0.099,-0.403)--(0.081,-0.492)--(0.067,-0.601)--(0.055,-0.734)--(0.045,-0.896)--(-0.045,-0.896)--(-0.055,-0.734)--(-0.067,-0.601)--(-0.081,-0.492)--(-0.099,-0.403)--(-0.121,-0.330)--(-0.148,-0.270)--(-0.181,-0.221)--(-0.221,-0.181)--(-0.270,-0.148)--(-0.330,-0.121)--(-0.403,-0.099)--(-0.492,-0.081)--(-0.601,-0.067)--(-0.734,-0.055)--(-0.896,-0.045)--(-0.896,0.045)--(-0.734,0.055)--(-0.601,0.067)--(-0.492,0.081)--(-0.403,0.099)--(-0.330,0.121)--(-0.270,0.148)--(-0.221,0.181)--(-0.181,0.221)--(-0.148,0.270)--(-0.121,0.330)--(-0.099,0.403)--(-0.081,0.492)--(-0.067,0.601)--(-0.055,0.734)--(-0.045,0.896)--cycle;
		\draw[thin](-1.3,0)--(1.3,0);
		\draw[thin](0,-1.3)--(0,1.3);
		\draw[thin,-stealth](0, 1.1)--(0, 1);
		\draw[thin,-stealth](0,-1.1)--(0,-1);
		\draw[thin,-stealth]( 1,0)--( 1.07,0);
		\draw[thin,-stealth](-1,0)--(-1.07,0);
	\end{tikzpicture}
	\caption{An $N$-capturing neighborhood $B$ of a hyperbolic fixed point. If $B$ is taken small enough, then the complement $B^\mathsf{c}$ is also $N$-capturing.}\label{f.bicapturing}
\end{figure}

\begin{proof}[Proof of \cref{l.capturingtower} when $\ell=1$]
Suppose $p$ is a fixed point of $T$.
Given $N>0$ and a neighborhood $U$ of $p$, 
take a smaller neighborhood $V$ such that 
\begin{equation}\label{e.smallness}
\bigcup_{i=-2N+1}^{2N-1} T^{-i} V \subseteq U \, .
\end{equation} 
Define sets
\begin{align}
V_\smallplus  &\coloneq \bigcap_{i=0}^{N-1} T^{-i} V \, , \qquad 
V_\smallminus \coloneq \bigcap_{i=0}^{N-1} T^{ i} V  = T^N V_\smallplus \, , \quad \text{and} \\ 
C &\coloneq \bigcup_{j=0}^{N-1} T^j V_\smallplus = \bigcup_{j=0}^{N-1} T^{-j} V_\smallminus \, . 
\end{align}
We have
\begin{equation}
C \setminus TC
= \bigcup_{j=0}^{N-1}  T^j V_\smallplus \setminus \bigcup_{j=1}^{N}  T^j V_\smallplus 
\subseteq V_\smallplus 
= \bigcap_{j=0}^{N-1}  T^{-j}(T^j V_\smallplus) 
\subseteq \bigcap_{j=0}^{N-1}  T^{-j}(C) \, ,
\end{equation}
that is, the set $C$ is $N$-capturing.

Define other sets 
\begin{align}\label{e.def_D}
B &\coloneq C \cup D \, , \quad \text{where}\\
D &\coloneq \bigcup_{\substack{i,j \ge 1 \\ i+j \le N}} (T^{-i} V_\smallplus \cap T^j V_\smallminus) \, . 
\end{align}
Note that 
\begin{equation}\label{e.Bellefonte}
T^{-1} D 
= \bigcup_{\substack{i \ge 2, \, j \ge 0 \\ i+j \le N}} (T^{-i} V_\smallplus \cap T^j V_\smallminus) 
\subseteq V_\smallplus \cup D 
\subseteq B \, .
\end{equation}
Then,
\begin{alignat}{2}
B \setminus TB 
&\subseteq C \setminus TB &\quad &\text{(since $B = C \cup D$ and $D \subseteq TB$)}\\
&\subseteq C \setminus TC &\quad &\text{(since $B \supseteq C$)}\\
&\subseteq \bigcap_{n=0}^{N-1} T^{-n} C &\quad &\text{(since $C$ is $N$-capturing)}\\
&\subseteq \bigcap_{n=0}^{N-1} T^{-n} B &\quad &\text{(since $B \supseteq C$),}
\end{alignat}
proving that $B$ is $N$-capturing. 

We will check that $B^\mathsf{c}$ is also $N$-capturing. 
A reasoning analogous to \eqref{e.Bellefonte} shows that $TD \subseteq B$.
On the other hand,
\begin{equation}\label{e.copycat}
C \setminus T^{-1}C
= \bigcup_{j=0}^{N-1}  T^{-j} V_\smallminus  \setminus \bigcup_{j=1}^{N}  T^{-j} V_\smallminus 
\subseteq V_\smallminus \, . 
\end{equation}
Then the following chain of inclusions holds:
\begin{alignat}{2}
B^\mathsf{c} \setminus TB^\mathsf{c} 
&= TB \setminus B \\
&\subseteq TC \setminus B &\quad &\text{(since $B = C \cup D$ and $TD \subseteq B$)}\\
&= T(C \setminus T^{-1}C) \setminus D \\ 
&\subseteq TV_\smallminus \setminus D &\quad &\text{(by \eqref{e.copycat}).} \label{e.basement}
\end{alignat}
It follows from \eqref{e.def_D} that 
\begin{equation}\label{e.afortiori}
D \supseteq TV_\smallminus \cap \bigcup_{i=1}^{N-1} T^{-i} V_\smallplus \, .
\end{equation}
Resuming from \eqref{e.basement}, we have the following inclusions:
\begin{alignat}{2}
B^\mathsf{c} \setminus TB^\mathsf{c} 
&\subseteq TV_\smallminus \setminus D  \\ 
&\subseteq \bigcap_{i=1}^{N-1} T^{-i} V_\smallplus^\mathsf{c} &\quad&\text{(by \eqref{e.afortiori})} \\ 
&\subseteq \bigcap_{i=1}^{N-1} T^{-i} B^\mathsf{c} &\quad&\text{(since $V_\smallplus \subseteq C \subseteq B$),} 
\end{alignat}
proving that $B^\mathsf{c}$ is $N$-capturing. 

Condition \eqref{e.smallness} ensures that $B \subseteq U$. 
Finally, if $X$ is zero-dimensional, then $V$ can be chosen as a clopen set, and it follows that $B$ is clopen. 
\end{proof}

\begin{proof}[Proof of \cref{l.capturingtower} for arbitrary $\ell$]
Suppose $p$ is a periodic point of least period~$\ell$.
Reducing the neighborhood $U$ if necessary, we can assume that $U$ is disjoint from $TU,\dots,T^{\ell-1} U$.
Given $N>0$, let $m \coloneq \lceil N/\ell \rceil$.
Consider the following smaller neighborhood of $p$:
\begin{equation}\label{e.def_V}
V \coloneq \bigcap_{j=0}^m T^{-\ell j} U \, .
\end{equation}
Applying the previously proved version of the \lcnamecref{l.capturingtower} (with $T^\ell$, $V$, and $m$ in the respective places of $T$, $U$, and $N$), we obtain a neighborhood $B$ of $p$ such that $B \subseteq V$, both $B$ and $B^\mathsf{c}$ are $m$-capturing with respect to $T^\ell$, and $B$ is clopen if $\dim X = 0$.
Let $K \coloneq \bigsqcup_{i=0}^{\ell - 1} T^i B$.
Note that 
\begin{equation}\label{e.rainy_afternoon}
B \subseteq \bigcap_{i=0}^{\ell - 1} T^{-i} K \, .
\end{equation}
We will check that both sets $K$ and $K^\mathsf{c}$ are $\ell m$-capturing (and a fortiori $N$-capturing) with respect to $T$.
We have
\begin{equation}
K \setminus TK
= \bigcup_{i=0}^{\ell-1} T^i B \setminus \bigcup_{i=1}^{\ell} T^i B \\
\subseteq B \setminus T^\ell B 
\subseteq \bigcap_{j=0}^{m-1} T^{-\ell j} B \, ,
\end{equation}
since $B$ is $m$-capturing with respect $T^\ell$.
Using \eqref{e.rainy_afternoon}, we obtain 
\begin{equation}
K \setminus TK 
\subseteq \bigcap_{j=0}^{m-1} \bigcap_{i=0}^{\ell - 1} T^{-\ell j - i} K 
= \bigcap_{n=0}^{\ell m-1} T^{-n} K \, ,
\end{equation}
showing that $K$ is $\ell m$-capturing.

Assume for a contradiction that $K^\mathsf{c}$ is not $\ell m$-capturing, that is, 
there exists $x \in K^\mathsf{c} \setminus TK^\mathsf{c}$ such that $T^n x \in K$ for some $n \in \ldbrack 0, \ell m - 1 \rdbrack$.
Then
\begin{equation}
x \in TK \setminus K 
= \bigcup_{i=1}^{\ell} T^i B \setminus \bigcup_{i=0}^{\ell-1} T^i B \\
\subseteq T^\ell B \setminus B \, .
\end{equation}
Use Euclidean division and write $n = \ell k + r$ with $k \in \ldbrack 0, m - 1 \rdbrack$ and $r \in \ldbrack 0, \ell - 1 \rdbrack$. 
Then,
\begin{equation}
T^n x 
\in T^{n+\ell} B 
\subseteq T^{n+\ell} V
= T^{\ell(k+1)+r} V
\subseteq T^r U \, ,
\end{equation}
by \eqref{e.def_V}.
The point $T^n x$ belongs to $K$, that is, it belongs to one of the sets $B , \dots, T^{\ell - 1} B$, which are respectively contained in the disjoint sets $U, \dots, T^{\ell-1} U$.
Hence the point $T^n x$ must belong to $T^r B$. That is, $T^{\ell k} x \in B$. However, $x \in B^\mathsf{c}$, therefore contradicting the fact that $B^\mathsf{c}$ is $m$-capturing with respect to $T^\ell$.
\end{proof}

\section{Castles with capturing towers} \label{s.dungeon}

Our new tower lemma announced in the introduction is stated as follows (where the capturing property is as in \cref{def.capturing}):

\begin{theorem}\label{t.dungeon}
Let $X$ be a compact metric space of zero topological dimension and let $T \colon X \to X$ be a homeomorphism.
Let $N>0$ and suppose that the set of periodic points of period less than $N$ is finite. 
Let $\delta>0$.
Then there exists a castle partition of $X$ into finitely many $N$-capturing towers whose floors are clopen sets of diameter less than~$\delta$.
\end{theorem}

\begin{proof}
If $T$ has no periodic points of period less than $N$, then by \cref{l.highcastle} there exists a partition of $X$ into finitely many towers of height at least $N$ whose floors are clopen sets. These towers are automatically $N$-capturing. Furthermore, we can slice the towers so that each floor has diameter less than $\delta$. This proves \cref{t.dungeon} in this particular case. 

Next, assume that $T$ has some periodic orbits of period less than $N$, but only finitely many of them: say $\mathcal{O}_1,\dots,\mathcal{O}_s$ (all distinct), with corresponding least periods $\ell_1,\dots,\ell_s$.
Let 
\begin{equation}\label{e.def_delta0}
\delta_0 \coloneq \min_{i\neq j} d(\mathcal{O}_i,\mathcal{O}_j) \, .
\end{equation}
(If $s=1$, then $\delta_0 \coloneq \infty$.)
Reducing $\delta$ if necessary, we can assume that $\delta < \delta_0/2$ and 
\begin{equation}\label{e.unif_cont}
\forall x,y \in X, \text{ if } d(x,y) < \delta, \text{ then }
\min_{n \in\ldbrack 1,N \rdbrack} d(T^n x, T^n y) < \frac{\delta_0}{2} \, .
\end{equation}
Applying \cref{l.capturingtower}, for each $i$ we can find a clopen set $B_i$ such that the tower
\begin{equation}\label{e.Ki}
K_i \coloneq B_i \sqcup TB_i \sqcup \cdots \sqcup T^{\ell_i - 1} B_i 
\end{equation}
contains $\mathcal{O}_i$, is $N$-capturing, and its complement $K_i^\mathsf{c}$ is also $N$-capturing.
Furthermore, $\diam B_i$ can be chosen sufficiently small so that each floor of $K_i$ has diameter less than $\delta$.
In particular, the set $K_i$ is contained in the $\delta_0/2$-neighborhood of the orbit $\mathcal{O}_i$. 
It follows from \eqref{e.def_delta0} that the sets $K_1,\dots,K_s$ are pairwise disjoint. 
Let 
\begin{equation}
K \coloneq \bigsqcup_{i=1}^s K_i \quad \text{and} \quad 
B \coloneq \bigsqcup_{i=1}^s B_i \, .
\end{equation}
So, $K$ is a castle with base $B$.
The set $K$ is $N$-capturing, since each of the towers~$K_i$ is $N$-capturing.

We claim that the complement set $K^\mathsf{c}$ is also $N$-capturing.
Indeed,	assume this assertion is false, that is, there exists a point $x \in K$ such that $Tx \not\in K$ and $T^n x \in K$ for some $n \in \ldbrack 2, N \rdbrack$.
Let $i$ and $j$ be such that $x \in K_i$ and $T^n x \in K_j$.
Since $K_i^\mathsf{c}$ is $N$-capturing, we must have $j \neq i$.
We have $d(x,\mathcal{O}_i) < \delta$ and so \eqref{e.unif_cont} ensures that $d(T^n x,\mathcal{O}_i) < \delta_0/2$. Furthermore, $d(T^n x,\mathcal{O}_j) < \delta <\delta_0/2$, so the triangle inequality yields $d(\mathcal{O}_i,\mathcal{O}_j)< \delta_0$, which contradicting \eqref{e.def_delta0}. 
This proves our claim: the set $K^\mathsf{c}$ is $N$-capturing.

Let $M$ be the maximal invariant set of $K^\mathsf{c}$, i.e.,
\begin{equation}
M \coloneq \bigcap_{n\in \Z} T^{-n}(K^\mathsf{c}) \, .
\end{equation}
Then $M$ is a (possibly empty) closed set and the restricted map $T|_M$ has no periodic points of period less than $N$. 
By \cref{l.highcastle}, $M$ admits a finite castle partition
\begin{equation}\label{e.castle_M}
M = \bigsqcup_{j=1}^t \bigsqcup_{n=0}^{m_j-1} T^n(D_j) \, ,
\end{equation}
into towers of heights $m_j \ge N$ whose bases $D_j$ are clopen sets in the relative topology of $M$.

Next, we enlarge the castle \eqref{e.castle_M}:
for each $i \in \ldbrack 1,t \rdbrack$, let $C_j$ be a clopen neighborhood of the compact set $D_j$. Choosing those neighborhoods small enough,
the following union is disjoint:
\begin{equation}\label{e.castle_L}
\bigsqcup_{j=1}^t \bigsqcup_{n=0}^{m_j-1} T^n(C_j) \eqcolon L \, ;
\end{equation}
furthermore, the resulting castle $L$ is disjoint from $K$.
Let $C \coloneq \bigsqcup_{j=1}^t C_j$ be the base of the castle $L$ (see \cref{fig.Wolfenstein}).

\begin{figure}
	\includegraphics{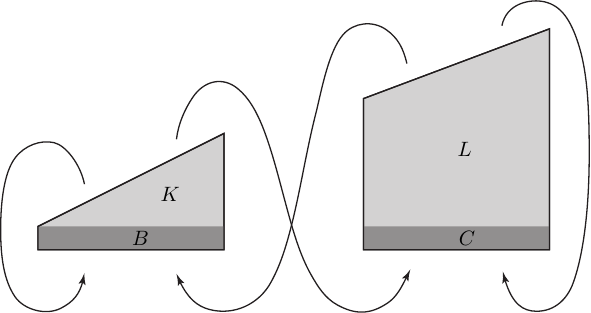}
	\caption{The two castles $K$ and $L$. The arrows denote possible excursions outside of $K\sqcup L$. Since the set $K^\mathsf{c}$ is $N$-capturing, excursions following the leftmost arrow take time at least $N$.}\label{fig.Wolfenstein}
\end{figure}

We claim that $B \sqcup C$ is a feedback set. 
By \cref{l.open_feedback} it is sufficient to check the (two-sided) orbit of every point $x \in X$ hits $B \sqcup C$. If $x \not \in M$, then, by definition of the set $M$, the orbit of $x$ hits $K$ and hence hits $B$. On the other hand, if $x \in M$, then the orbit of $x$ never hits the invariant castle $M$, and so it hits its base, which is contained in $C$. 
This proves our claim; actually, this argument shows that $B \sqcup (C \cap T(L))$ is also a feedback set.

Define another clopen set: 
\begin{equation}\label{e.simpler_base}
E \coloneq B \sqcup (T(L) \cap C) \sqcup (T(K) \setminus K) \, . 
\end{equation}
Then $E$ is a feedback set, since it contains the feedback set $B \sqcup (C \cap T(L))$.
Let $Q$ be the Kakutani--Rokhlin castle with base $E$ given by \cref{l.Kakutani}; this is a partition of the whole space $X$ into finitely many towers with clopen floors. 
Note that each tower $K_i$ of the castle $K$ is contained in a tower of the same height $\ell_i$ of the new castle $Q$; indeed, if $x$ is in the base $B_i$ of the tower $K_i$, then $T^{\ell_i}(x)$ belongs to $B \sqcup (T(K) \setminus K)$ and hence to $E$, the base of $Q$.  

Therefore we can refine the castle $Q$ (by slicing the towers that intersect $K$) and obtain a castle $Q'$ with the same base $E$ such that every tower of $K$ is also a tower of $Q'$. Recall that each of these towers have heights less than $N$, are $N$-capturing, and have floors of diameters less than $\delta$. By further slicing of towers, we can assume that all floors of all towers of $Q'$ have diameter less than $\delta$.

To complete the proof of the \lcnamecref{t.dungeon}, it remains to check that each tower of $Q'$ which is not a tower of $K$ has height at least $N$. Equivalently, for each point $x$ in the set $E \setminus B = (T(L) \cap C) \sqcup (T(K) \setminus K)$, the segment of orbit $Tx,T^2 x,\dots,T^{N-1} x$ does not hit $E$.

First, consider the case where $x \in T(L) \cap C$. Then the segment of orbit above is contained in $L \setminus C$ and thus in $E^\mathsf{c}$ (note that $(T(K) \setminus K) \cap L \subseteq L \setminus T(L) \subseteq C$). 

Next, consider the remaining case where $x \in T(K) \setminus K$. Consider the return time to $E$, that is, the least positive $n>0$ such that $T^n x \in E$. The point $T^n x$ cannot belong to $T(K) \setminus K$, otherwise there would be an earlier hit to $B \subseteq E$. Therefore, $T^n x \in B \sqcup (T(L) \cap C)$. If $T^n x$ is in $B$, then the fact that the set $K^\mathsf{c}$ is $N$-capturing tells us that $n \ge N$, as desired. If, on the other hand, $T^n x$ is in $T(L) \cap C$, then the point $T^{n-1} x$ is in the top floor of a tower of $L$, and since $x \not\in L$, the segment of orbit $Tx,T^2 x,\dots,T^{n-1} x$ must include a full traverse of that tower. In particular, $n$ is bigger than the height of the tower, which is at least $N$. 

This completes the proof of the \lcnamecref{t.dungeon}.
\end{proof}

\section{Eradicating quasiconformal orbits}\label{s.eradicate}

In this section, we use the previous results about castles to prove our main goal, \cref{t.main}. 
We need a few preliminaries in linear algebra.

The \emph{spectral radius} and the \emph{eigenvalue condition number} of a square matrix $A$ are defined respectively as:
\begin{align}
\label{e.def_rho}
\rho(A) &\coloneq \max \{|\lambda| \st \lambda \in \C \text{ is an eigenvalue of $A$}\} \, , \\ 
\label{e.def_kappa_e}
\kappa_\mathrm{e}(A) &\coloneq \rho(A) \rho(A^{-1}) \, .
\end{align}
Note that 
$\rho(A^n) = \rho(A)^n$ for every $n > 0$ and 
$\kappa_\mathrm{e}(A^n) = \kappa_\mathrm{e}(A)^n$ for every $n \neq 0$.

\begin{lemma}\label{l.eigenval_perturb}
Let $\mathbb{F}$ be either $\C$ or $\R$ and let $d \ge 2$.
Assume that $d\ge 3$ if $\mathbb{F} = \R$.
Let $A_0, \dots, A_{\ell-1}$ be a finite list of matrices in $\mathrm{GL}(d,\mathbb{F})$.
Then, for every $\epsilon>0$, there exists $\tilde A_0, \dots, \tilde A_{\ell-1} \in \mathrm{GL}(d,\mathbb{F})$ such that $\|\tilde{A}_j - A_j \| \le \epsilon \|A_j\|$ for each $j \in \ldbrack 0, \ell - 1 \rdbrack$ and 
\begin{equation}\label{e.big_kappa_e}
\kappa_\mathrm{e} \big( \tilde{A}_{\ell-1} \cdots \tilde{A}_0 \big) \ge (1+\epsilon)^\ell \, .
\end{equation}
\end{lemma}

\begin{proof}
Let $\lambda_1,\dots,\lambda_d$ be the eigenvalues of the matrix $A_{\ell-1} \cdots A_0$, ordered so that $|\lambda_1| \ge \cdots \ge |\lambda_d|$ and furthermore $\lambda_{d-1} = \overline{\lambda_d}$ in the case that $\mathbb{F} = \mathbb{R}$ and $\lambda_d \in \C \setminus \R$.
If $\mathbb{F} = \mathbb{R}$ and $\lambda_d \in \C \setminus \R$, let $\nu \coloneq 2$, otherwise let $\nu \coloneq 1$.
Let $V_0 \subset \mathbb{F}^d$ be an $\nu$-dimensional eigenspace of $A_{\ell-1} \cdots A_0$ corresponding to the eigenvalue $\lambda_d$ if $\nu=1$, or to the eigenvalue pair $\lambda_{d-1} , \lambda_d$ if $\nu=2$.
For each $j \in \ldbrack 1, \ell-1 \rdbrack$, define the $\nu$-dimensional subspace
$V_i \coloneq A_{i-1} \cdots A_0 (V_0)$. 
Let $\{e_1,\dots,e_d\}$ denote the canonical basis of $\mathbb{F}^d$.
For each $j \in \ldbrack 0, \ell-1 \rdbrack$, take an orthogonal (if $\mathbb{F} = \R$) or unitary (if $\mathbb{F}=\C$) matrix $R_j$ such that $R_j V_j$ equals $\mathrm{span}(e_d)$ if $\nu=1$ and $\mathrm{span}(e_{d-1},e_d)$ if $\nu=2$.
Let $R_\ell \coloneq R_0$.
For each $j \in \ldbrack 0, \ell-1 \rdbrack$, let $\Delta_j \coloneq R_{j+1} A_j R_j^{-1}$. 
Then $\Delta_j$ is a block triangular matrix:
\begin{equation}
\Delta_j \eqcolon
\begin{pmatrix} 
	B_j	&	0	\\
	T_j	&	C_j
\end{pmatrix} \, , \quad \text{where $C_j$ is a $\nu \times \nu$ matrix.}
\end{equation}
Multiplying these matrices, we get another block diagonal matrix
\begin{equation}
\Delta_{\ell-1} \cdots \Delta_0 = 
\begin{pmatrix} 
	B_{\ell-1} \cdots B_0	&	0	\\
	*						&	C_{\ell-1} \cdots C_0
\end{pmatrix} 
\end{equation}
Since $\Delta_{\ell-1} \cdots \Delta_0 = R_0 A_{\ell-1} \cdots A_0 R_0^{-1}$, the submatrix $B_{\ell-1} \cdots B_0$ has eigenvalues $\lambda_1,\dots,\lambda_{d-\nu}$, while the submatrix $C_{\ell-1} \cdots C_0$ has eigenvalues $\lambda_d$ if $\nu=1$ or $\lambda_{d-1},\lambda_d$ if $\nu=2$.

For each $j \in \ldbrack 0, \ell-1 \rdbrack$, define matrices
\begin{equation}\label{e.epsilon_perturb}
\tilde{\Delta}_j \coloneq 
\begin{pmatrix} 
	(1+\epsilon) B_j	&	0	\\
	T_j					&	C_j
\end{pmatrix}
\quad\text{and}\quad
\tilde{A}_j \coloneq R_{j+1}^{-1} \tilde{\Delta}_j R_j \, .
\end{equation}
Note that
\begin{equation}\label{e.norm_bound}
\|\tilde{A}_j - A_j \| =
\|\tilde{\Delta}_j - \Delta_j \| =
\left\| \left(\begin{smallmatrix} \epsilon I_{d-\nu} & 0 \\ 0 & 0 \end{smallmatrix} \right) \Delta_j \right\| \le
\epsilon \|\Delta_j\| = 
\epsilon \|A_j\| \, .
\end{equation}
Multiplying the block-triangular matrices $\tilde{\Delta}_j$, we obtain
\begin{equation}
\tilde{\Delta}_{\ell-1} \cdots \tilde{\Delta}_0 = 
\begin{pmatrix} 
	(1+\epsilon)^\ell B_{\ell-1} \cdots B_0	&	0	\\
	*										&	C_{\ell-1} \cdots C_0
\end{pmatrix} \, . 
\end{equation}
This matrix has eigenvalues 
\begin{alignat}{2}
&(1+\epsilon)^\ell \lambda_1 , \ \dots, \ (1+\epsilon)^\ell \lambda_{d-1}, \ \lambda_d &\quad &\text{if $\nu=1$, or } \\ 
&(1+\epsilon)^\ell \lambda_1 , \ \dots, \ (1+\epsilon)^\ell \lambda_{d-2}, \ \lambda_{d-1}, \ \lambda_d &\quad &\text{if $\nu=2$.}
\end{alignat}
In any case, 
\begin{equation}
\kappa_{\mathrm{e}} \big( \tilde{\Delta}_{\ell-1} \cdots \tilde{\Delta}_0 \big) = \frac{(1+\epsilon)^\ell |\lambda_1|}{|\lambda_d|} \ge (1+\epsilon)^\ell \, .
\end{equation}
Since the matrix $\tilde{\Delta}_{\ell-1} \cdots \tilde{\Delta}_0$ is similar (i.e.\ conjugate) to $\tilde{A}_{\ell-1} \cdots \tilde{A}_0$, inequality \eqref{e.big_kappa_e} follows. 
\end{proof}

Recall that the condition number of an invertible $d \times d$ real or complex matrix $A$ was defined in \eqref{e.def_kappa} as $\kappa(A) \coloneq \|A\| \, \|A^{-1}\| $, where $\| \mathord{\cdot} \|$ denotes the Euclidean-induced operator norm. 
The norm and the condition number can be compared to the spectral radius \eqref{e.def_rho} and to the eigenvalue condition number \eqref{e.def_kappa_e} as follows:
\begin{equation}
\|A\| \ge \rho(A) \, , \quad 
\kappa(A) \ge \kappa_{\mathrm{e}}(A) \, .
\end{equation}
Since the norm is sub-multiplicative, so is the condition number:
\begin{equation}\label{e.submult}
	\kappa(AB) \le \kappa(A) \kappa(B) \, .
\end{equation}
In terms of the singular values $\sigma_1(A) \ge \cdots \ge \sigma_d(A)$ of the matrix $A$, we have $\|A\|=\sigma_1(A)$ and $\kappa(A) = \sigma_1(A) / \sigma_d(A)$.

\begin{lemma}\label{l.sing_val_perturb}
Let $\mathbb{F}$ be either $\C$ or $\R$, let $d \ge 2$, and let $A_0, \dots, A_{\ell-1}$ be a finite list of matrices in $\mathrm{GL}(d,\mathbb{F})$.
Then, for every $\epsilon>0$, there exists $\tilde A_0, \dots, \tilde A_{\ell-1} \in \mathrm{GL}(d,\mathbb{F})$ such that $\|\tilde{A}_j - A_j \| \le \epsilon \|A_j\|$ for each $j \in \ldbrack 0, \ell - 1 \rdbrack$ and 
\begin{equation}\label{e.big_kappa}
\kappa \big( \tilde{A}_{\ell-1} \cdots \tilde{A}_0 \big) \ge (1+\epsilon)^\ell \, .
\end{equation}
\end{lemma}

\begin{proof}
It is sufficient to consider the case where $\mathbb{F} = \R$ and $d = 2$, since the other cases follow immediately from \cref{l.eigenval_perturb}.
Let $v_0 \in \R^2$ be a unit vector such that $\|A_{\ell-1} \cdots A_0 v_0 \| = \sigma_2(A_{\ell-1} \cdots A_0)$.
For each $j \in \ldbrack 1, \ell \rdbrack$, define the unit vector
\begin{equation}
v_j \coloneq \frac{A_{i-1} \cdots A_0 v_0}{\|A_{i-1} \cdots A_0 v_0\|}
\end{equation}
and take an orthogonal matrix $R_j$ such that $R_j v_j$ equals $e_2 \coloneq (0,1)$.
For each $j \in \ldbrack 0, \ell-1 \rdbrack$, let $\Delta_j \coloneq R_{j+1} A_j R_j^{-1}$. 
Then $\Delta_j$ is a lower-triangular matrix:
\begin{equation}
\Delta_j \eqcolon
\begin{pmatrix} 
	b_j	&	0	\\
	t_j	&	c_j
\end{pmatrix} \, .
\end{equation}
Multiplying those matrices, we get a lower-triangular matrix $\Delta_{\ell-1} \cdots \Delta_0$, which equals $R_\ell A_{\ell-1} \cdots A_0 R_0^{-1}$ and so it has the same singular values as $A_{\ell-1} \cdots A_0$.
It follows that this matrix product is actually diagonal:
\begin{equation}
\Delta_{\ell-1} \cdots \Delta_0 = 
\begin{pmatrix} 
	b_{\ell-1} \cdots b_0	&	0	\\
	0						&	c_{\ell-1} \cdots c_0
\end{pmatrix} 
\end{equation}
with $|b_{\ell-1} \cdots b_0| = \sigma_1(A_{\ell-1} \cdots A_0)$ and $|c_{\ell-1} \cdots c_0| = 
\sigma_2(A_{\ell-1} \cdots A_0)$.

Analogously to \eqref{e.epsilon_perturb}, for each $j \in \ldbrack 0, \ell-1 \rdbrack$ we define matrices
\begin{equation}
\tilde{\Delta}_j \coloneq 
\begin{pmatrix} 
	(1+\epsilon) b_j	&	0	\\
	t_j					&	c_j
\end{pmatrix}
\quad\text{and}\quad
\tilde{A}_j \coloneq R_{j+1}^{-1} \tilde{\Delta}_j R_j \, .
\end{equation}
By the previous computation \eqref{e.norm_bound}, we have $\|\tilde{A}_j - A_j \| \le \epsilon \|A_j\|$.
Multiplying the lower-triangular matrices $\tilde{\Delta}_j$, we obtain
\begin{equation}
\tilde{\Delta}_{\ell-1} \cdots \tilde{\Delta}_0 = 
\begin{pmatrix} 
	(1+\epsilon)^\ell b_{\ell-1} \cdots b_0	&	0	\\
	*										&	c_{\ell-1} \cdots c_0
\end{pmatrix} \, . 
\end{equation}
Then 
\begin{alignat}{2}
\sigma_1 (\tilde{\Delta}_{\ell-1} \cdots \tilde{\Delta}_0) &\ge \| \tilde{\Delta}_{\ell-1} \cdots \tilde{\Delta}_0 e_1 \| &&\ge (1+\epsilon)^\ell  |b_{\ell-1} \cdots b_0| \, , \\
\sigma_2 (\tilde{\Delta}_{\ell-1} \cdots \tilde{\Delta}_0) &\le \| \tilde{\Delta}_{\ell-1} \cdots \tilde{\Delta}_0 e_2 \| &&= |c_{\ell-1} \cdots c_0| \, .
\end{alignat}
Therefore,
\begin{equation}
\kappa(\tilde{A}_{\ell-1} \cdots \tilde{A}_0) 
=  \kappa(\tilde{\Delta}_{\ell-1} \cdots \tilde{\Delta}_0) 
\ge \frac{(1+\epsilon)^\ell |b_{\ell-1} \cdots b_0|}{|c_{\ell-1} \cdots c_0|} 
\ge (1+\epsilon)^\ell \, . \qedhere
\end{equation}
\end{proof}

We are ready to prove our main result.

\begin{proof}[Proof of \cref{t.main}]
Assume $X$ is a zero-dimensional compact metric space and that $T \colon X \to X$ is a homeomorphism with finitely many periodic points of any given period. Let $\mathbb{F} = \C$ or $\R$. If $\mathbb{F} = \R$ and $T$ has periodic points, assume that $d \ge 3$; otherwise, assume that $d \ge 2$.
For each $M > 1$, let
\begin{equation}
\mathcal{U}_M \coloneq \left\{ F \in C^0(X, \mathrm{GL}(d,\mathbb{F})) \st \forall x \in X, \ \sup_{n \in \Z} \kappa(F^{(n)}(x)) > M \right\} . 
\end{equation}
A compactness argument shows that $\mathcal{U}_M$ is open in $C^0(X, \mathrm{GL}(d,\mathbb{F}))$.
We will establish that $\mathcal{U}_M$ is dense.
From this it will follow that the set $\mathcal{QC} = \bigcap_M \mathcal{U}_M^\mathsf{c}$ is meager, therefore completing the proof. 

Instead of working with the metric \eqref{e.def_metric}, we use instead the metric
\begin{equation}
	d'(F,G) \coloneq \sup_{x \in X} \|F(x)-G(x)\| \, ,
\end{equation}
which, despite not being complete, induces the same topology on $C^0(X, \mathrm{GL}(d,\mathbb{F}))$.

Fix $M>1$, $F \in C^0(X, \mathrm{GL}(d,\mathbb{F}))$, and $\epsilon > 0$. 
Let $C_0 \coloneq \sup_{x \in X} \|F(x)\|$.
Fix $N>0$ such that 
\begin{equation}\label{e.N}
(1+\epsilon)^N > M^2 \, . 
\end{equation}
Let $\delta>0$ be such that for all $x,y \in X$, if $d(x,y)<\delta$, then $\|F(x)-F(y)\|<\epsilon$.

By \cref{t.dungeon}, there exists a castle partition 
\begin{equation}
	X = \bigsqcup_{i=1}^s \bigsqcup_{j=0}^{\ell_i - 1} T^j B_i \, .
\end{equation}
where each tower $K_i \coloneq \bigsqcup_{j=0}^{\ell_i - 1} T^j B_i$ is $N$-capturing and each floor $T^j B_i$ is a clopen set of diameter $< \delta$.
It follows from the proof of \cref{t.dungeon} (first paragraph) that if $T$ has no periodic points, then $\ell_i \ge N$ for each $i$.

For each $i \in \ldbrack 1, s \rdbrack$ and $j \in \ldbrack 0, \ell_i-1 \rdbrack$, choose an arbitrary point $x_{i,j}$ in the set $T^i B_j$ and let $A_{i,j} \coloneq F(x_{i,j})$.
Using \cref{l.eigenval_perturb,l.sing_val_perturb}, we find matrices $\tilde{A}_{i,j}$ such that $\| \tilde{A}_{i,j} - A_{i,j}\| \le \epsilon \|A_{i,j}\|$ and
\begin{alignat}{2}
\kappa_\mathrm{e}(\tilde{A}_{i,\ell_i-1} \cdots \tilde{A}_{i,0}) &\ge (1+\epsilon)^{\ell_i} &\quad \text{if } \ell_i &< N \, , \\ 
\kappa           (\tilde{A}_{i,\ell_i-1} \cdots \tilde{A}_{i,0}) &\ge (1+\epsilon)^{\ell_i} &\quad \text{if } \ell_i &\ge N \, .
\end{alignat}
Note that if $\mathbb{F} = \R$ and $d = 2$, then \cref{l.eigenval_perturb} is not applicable; but in this case $T$ has no periodic points, so $\ell_i \ge N$ for all $i$ and only \cref{l.sing_val_perturb} is required.

Now define a function $G \colon X \to \mathrm{GL}(d,\mathbb{F})$ as constant equal to $\tilde{A}_{i,j}$ on the set $T^i(B)$.
Since those sets are clopen, $G$ is continuous. 
Let us check that $G$ is a perturbation of $F$.
Every point $x \in X$ belongs to a unique element $T^j(B_i)$ of the castle partition.
Then we can bound:
\begin{align}
\|G(x) - F(x) \| 
&\le \|\tilde{A}_{i,j} - A_{i,j}\| + \|F(x_{i,j}) - F(x)\| \\ 
&< C \epsilon + \epsilon \, .
\end{align}
Thus, $d'(G,F) < (C+1)\epsilon$.

The final task is checking that $G \in \mathcal{U}_M$, that is, for every  $x\in X$,
\begin{equation}\label{e.final}
\exists n \in \Z \text{ s.t.\ }
\kappa(G^{(n)}(x)) > M \, .
\end{equation}
So fix a point $x \in X$. 
Let $K_i$ be the tower containing $x$. 

As a first case, assume that the height $\ell_i$ of the tower is at least $N$.
We have $x = T^j y$ for some $y \in B_i$ and $j \in \ldbrack 0, \ell_i -1 \rdbrack$.
Then $\kappa(G^{(\ell_i)}(y)) \ge (1+\epsilon)^{\ell_i} > M^2$.
Since
\begin{equation}
	G^{(\ell_i)}(y) = G^{(\ell_i-j)}(x) \, [G^{(-j)}(x)]^{-1} \, ,
\end{equation}
it follows from the sub-multiplicativity property~\eqref{e.submult} that either $\kappa(G^{(\ell_i-j)}(x))$ or $\kappa(G^{(-j)}(x))$ is bigger than $M$.
This proves property \eqref{e.final} in this case.

Now assume that $\ell_i < N$.
Since $K_i$ is an $N$-capturing set and also a tower of height $\ell_i$, it is actually a $m_i \ell_i$-capturing set, where $m_i \coloneq \lceil N/\ell_i \rceil$. 
Using characterization \eqref{e.capture_reformulated} of the capturing property, we see that
there exists a point $y \in B_i$ such that the segment of orbit $\{T^n y \st n \in \ldbrack 0, m_i \ell_i-1\rdbrack\}$ is contained in $K_i$ and contains the point $x$. 
Then $\kappa_{\mathrm{e}}(G^{(\ell_i)}(y)) \ge (1+\epsilon)^{\ell_i}$; furthermore, $G^{(m_i \ell_i)}(y) = [G^{(\ell_i)}(y)]^{m_i}$, so
\begin{equation}
\kappa(G^{(m_i \ell_i)}(y)) \ge \kappa_{\mathrm{e}}(G^{(m_i \ell_i)}(y)) = [\kappa_{\mathrm{e}}(G^{(\ell_i)}(y))]^{m_i} \ge (1+\epsilon)^{m_i \ell_i} > M^2 \, .
\end{equation}
Take $k \in \ldbrack 0, m_i \ell_i-1\rdbrack$ such that $T^k y = x$.
Using sub-multiplicativity of $\kappa$ once again, we conclude that either $\kappa(G^{(m_i\ell_i-k)}(x))$ or $\kappa(G^{(-k)}(x))$ 
is bigger than $M$. 
This concludes the proof of \eqref{e.final}.
That is, $G \in \mathcal{U}_M$.

We have shown that the open set $\mathcal{U}_M$ is dense. 
The proof of \cref{t.main} is complete.
\end{proof}

\section{Further comments}  \label{s.future}

The possibility of relaxing the hypotheses of \cref{t.main,t.dungeon} requires further investigation. It would be especially interesting to know if the results remain valid if we only require the space $X$ to have \emph{finite} topological dimension. (Of course, the requirement that the floors are clopen in \cref{t.dungeon} would have to be relaxed as well.)

We note that Downarowicz's \cite[Lemma~1]{Down} (which we mentioned in \cref{s.castles} above) admits an extension to spaces of any finite dimension, due to Gutman \cite[Theorem~6.1] {Gutman}. 
However, Gutman's theorem is false in infinite-dimensional spaces: see \cite[Theorem~1.1]{TTY}.
Let us also point out a related tower theorem of Bonatti and Crovisier \cite[Theorem~3.1]{BonCro}, which applies to diffeomorphisms of finite-dimensional manifolds and allows for periodic points (with some technical restrictions).

Towers in finite-dimensional spaces have been used for perturbing linear cocycles in the papers \cite{AvB,ABD,B_Quart}. Some arguments from \cite{ABD} were later simplified without the use of towers: see \cite{BNavas}.
Although we do not foresee a proof of \cref{t.main} (or any substantive variation) that circumvents the use of towers, this may be a possibility worth considering.


\bigskip

\noindent \textbf{Acknowledgements.} 
I thank Meysam Nassiri, Hesam Rajabzadeh, and Zahra Reshadat for enlightening conversations. 
I am grateful to Sylvain Crovisier for bring reference \cite{Gutman} to my attention.
I thank the referee for a number of corrections and suggestions that helped improve the exposition.


\end{document}